\documentclass[12pt,a4paper]{amsart}
\usepackage{amssymb}
\usepackage{amsfonts}
\usepackage{amsmath}
\usepackage[mathscr]{eucal}
\usepackage{enumerate}
\usepackage{color,graphicx,url}

\usepackage{comment}

\newtheorem{theorem}{Theorem}[section]
\newtheorem{corollary}[theorem]{Corollary}

\newtheorem{prop}[theorem]{Proposition}

\newtheorem{remark}[theorem]{Remark}

\renewenvironment{proof}%
	{\smallbreak\noindent {\bf Proof.\ }}%
	{\hskip 1em \rule[-2pt]{5pt}{10pt}

 \medbreak}
\def\1{\mbox{\boldmath $1$}} 
\def\a{\mbox{\boldmath $a$}}

\def\A{\mathbf{A}}
\def\B{\mathbf{B}}
\def\cs{\mathrm{cs}}
\def\D{\mathbf{D}}
\def\F{\mathbf{F}}
\def\M{\mathbf{M}}
\def\O{\mathbf{O}}
\def\P{\mathbf{P}}
\def\Q{\mathbf{Q}}
\def\I{\mathbf{I}}
\def\J{\mathbf{J}}
\def\s{\mathbf{s}}
\def\T{\mathbf{T}}

\def\C{\mathbb{C}}
\def\R{\mathbb{R}}
\def\Z{\mathbb{Z}}

\def\Tr{{\mathrm{Tr}}}

\renewcommand{\Re}{\mathrm{Re}}
\renewcommand{\Im}{\mathrm{Im}}
\def\Spec{{\mathrm{Spec}}}

\title[short title]{THE LIMIT THEOREM WITH RESPECT TO THE MATRICES ON NON-BACKTRACKING PATHS OF A GRAPH}
\title[The matrices on non-backtracking paths of a graph]{
The limit theorem with respect to the matrices on non-backtracking paths of a graph
}

\author{Takehiro Hasegawa}
\address[T. Hasegawa]{Department of Education, Shiga University, Otsu, Shiga 520-0862, JAPAN}
\email{thasegawa3141592@yahoo.co.jp}
\author{Takashi Komatsu}
\address[T. Komatsu]{Math.~Research Institute Calc for Industry, Minami, Hiroshima, 732-0816, JAPAN}
\email{ta.komatsu@sunmath-calc.co.jp}
\author{Norio Konno}
\address[N. Konno]{Department of Applied Mathematics, Faculty of Engineering, Yokohama National University, Hodogaya, Yokohama 240-8501, JAPAN}
\email{konno@ynu.ac.jp}
\author{Hayato Saigo}
\address[H. Saigo]{Nagahama Institute of Bio-Science and Technology, 1266, Tamura, Nagahama 526-0829, JAPAN}
\email{h\_saigoh@nagahama-i-bio.ac.jp}
\author{Seiken Saito}
\address[S. Saito]{
Division of Liberal Arts, 
Center for Promotion of Higher Education, 
Kogakuin University,
2665-1 Nakano, Hachioji, Tokyo 192-0015, JAPAN
}
\email{saito.seiken@cc.kogakuin.ac.jp}
\author{Iwao Sato}
\address[I. Sato]{Oyama National College of Technology, Oyama, Tochigi 323-0806, JAPAN}
\email{isato@oyama-ct.ac.jp}
\author{Shingo Sugiyama}
\address[S. Sugiyama]{Department of Mathematics, College of Science and Technology, Nihon University, Suruga-Dai, Kanda, Chiyoda, Tokyo 101-8308, JAPAN}
\email{sugiyama.shingo@nihon-u.ac.jp}

\date{\today}
\keywords{
non-backtracking paths, regular graphs, arcsine law\hspace{-0.05pt}}

\begin{document}
\begin{abstract}
We give a limit theorem with respect to the matrices related to non-backtracking paths 
of a regular graph. The limit obtained closely resembles the $k$th moments of the arcsine law. 
Furthermore, we obtain the asymptotics of the averages of 
the $p^m$th Fourier coefficients of the cusp forms related to  
the Ramanujan graphs defined by A. Lubotzky, R. Phillips and P. Sarnak.  
\end{abstract}

\maketitle

\section{Introduction}
In spectral graph theory related to probability theory and number theory, it is important to investigate the relations between non-backtracking paths in a graph and eigenvalues of its adjacency matrix (cf.\ \cite{Serre97, ABLS, Friedman}). 
In fact, Bordenave, Lelarge and Massouli\'e \cite{BLM2015, BLM2018} studied 
the spectrum of the {\it non-backtracking matrix} (or equivalently, {\it edge matrix}) 
for a generalization of the Erd\"os--R\'enyi graphs, 
and they proved that the generalization satisfies
{\it Graph theory Riemann hypothesis} with high probability
(cf.\ [5, Theorems 3 and 4]. See \cite{Terras-Stark07} or \cite[Eq.(8.6)]{Terras2011} 
for the definition of Graph theory Riemann hypothesis).
Stark and Terras proved that a regular graph satisfies Graph theory Riemann hypothesis 
if and only if it is a Ramanujan graph (cf.\ \cite[Corollary 1]{Stark-Terras96}).
This indicates that the Erd\"os--R\'enyi graph can be regarded
as an ``irregular" Ramanujan graph (cf.\ \cite{Terras-Stark07, Terras2011, Saito}). 
The purpose of this paper is to define and analyze a new kind of matrix to investigate relations between the non-backtracking paths without tails in a finite connected regular graph and the eigenvalues of its adjacency matrix. 

Let $G$ be a finite connected (possibly irregular) graph, 
and let $N_m$ be the number of non-backtracking ``closed'' paths without tails (or equivalently, reduced cycles) in $G$ of length $m$ (see \eqref{Nm} below). 
The sequence $\{N_m\}_{m \ge 1}$ can be considered as a kind of  ``gap'' which indicates how $G$ is far from the universal covering tree of $G$, 
since $G$ is a tree if and only if $N_m=0$ for all $m$. 
To capture all $N_m$, the {\it Ihara zeta function} of $G$ is useful. 
The Ihara zeta function of $G$ is the generating function of $\{ N_m \}_{m \geq 1}$: 
\begin{align}\label{zeta-exp}
Z_G(u)=\exp\left( \sum_{m\ge 1} \dfrac{N_m}{m}u^m\right)
\end{align}
for $u\in \C$ such that $|u|$ is sufficiently small.
Note that $G$ is a tree if and only if $Z_G(u)=1$. 

In 1966, Ihara \cite{Ihara} defined the function $Z_G(u)$ for any regular graph $G$,
and proved that the reciprocal of $Z_G(u)$ is an explicit polynomial. 
After that, Hashimoto \cite{Hashimoto} and Bass \cite{Bass} extended Ihara's result to any graph (see Theorem \ref{Ihara-Bass} below). 
In 1996, Stark and Terras \cite{Stark-Terras96} introduced a matrix $\M_m$ related to the numbers of non-backtracking paths in any graph (cf.\ \cite[(2.7)]{Stark-Terras96}), 
and gave an elementary proof of Ihara's result for any graph. 
Moreover, Stark and Terras showed that
\[
\mathrm{Tr}(\M_m)=N_m
\]
for any graph. By this equation, $\M_m$ can be considered as a matrix generalization of $N_m$. 
Note that, for a regular graph, Ihara had already considered the matrix $\M_m$ 
(which appears as  the matrix $A^\chi_m-(q-1)\sum_{k=1}^{[m/2]}A^\chi_{m-2k}$ in \cite[p.229]{Ihara}).
To the best of our knowledge, the matrix $\M_m$ itself has not been sufficiently analyzed or discussed with respect to its explicit expressions, inequalities, asymptotic behaviors, distributions, etc.

In this paper, for regular graphs, we analyze the matrix $\M_m$ and its variant
(written as $\a_m$). 
More precisely, we define a new matrix $\a_m$ by using $\M_m$ for a regular graph $G$. 
From the point of view of graph spectrum, we call $\a_m$ the {\it principal part} of  $\frac{1}{2q^{m/2}}\{\M_m-e_m(q-1)\I_n\}$ 
with respect to the adjacency matrix $\A(G)$ (see \S 4 below). 
To investigate $\M_m$ as a generalization of $N_m$ which plays a crucial role in spectral graph theory, 
we focus on the asymptotic behavior of the principal part $\a_m$ of $\frac{1}{2q^{m/2}}\{\M_m-e_m(q-1)\I_n\}$. 
In order to study the ``distribution" of $\a_m$, we consider the limit theorem of its ``$k$th moments." 
We also obtain the asymptotic behavior of the matrices $\s_m$ related to $\a_m$. 
As an application,  
we give an asymptotic formula of the average of $a(p^m)$, where $a(p^m)$ is 
the $p^m$th Fourier coefficient of the weight $2$ cusp form related to 
LPS Ramanujan graph $X^{p,q}$ (cf.\ \cite{LPS}).
This cusp form is the cuspidal part of 
the theta series which comes from the Hamilton quaternion algebra. 
It is very interesting that the $p^m$th Fourier coefficients of this theta series can be written by using the numbers of non-backtracking closed paths on $X^{p,q}$ (see Remark \ref{LPS} below).

Before ending the introduction, we state several remarks on statistical studies on graphs.
For a general regular graph, 
the relation between 
geometric objects such as the numbers of closed paths and 
eigenvalues of the adjacency matrix are well investigated, 
in spectral graph theory related to probability theory and number theory.
In 1987, Ahumada \cite{Ahumada} showed the Selberg trace formula
(STF for short) for regular graphs (see Theorem \ref{STF} below), 
and since then many researchers discuss the STF (cf.\ \cite{VN,Terras1999, Mnev}). 
Due to results of harmonic analysis on regular trees, 
the STF of a regular graph $G$ gives a beautiful relation between the eigenvalues of the adjacency matrix $\A(G)$ and the numbers $\{N_m\}_{m \ge 1}$. 
In 1981, McKay \cite{McKay} determined the limiting probability density for the eigenvalues of 
a series of regular graphs (the so-called Kesten-McKay law).  
In \cite{LPS}, Lubotzky, Phillips and Sarnak proved the Kesten-McKay law by using some kind of 
the STF for the sequence of their Ramanujan graphs.
We remark that Ramanujan graphs were explicitly constructed in \cite{Chiu} and \cite{JSY} by quaternion algebras other than the Hamilton quaternion algebra in the same method of \cite{LPS}.

This paper is organized as follows. 
In Section 2, we introduce the terminology of graph theory which are used in Sections $3$ and $4$. 
By explaining the Ihara zeta function $Z_{G}(u)$ of a graph $G$ and the Selberg trace formula of regular graphs, 
we present short reviews for non-backtracking closed paths in $G$ and for the numbers $N_m$. 
In Section 3, we restrict a graph $G$ to a regular graph, 
and we discuss several properties of the matrices $\M_m$. 
In Section 4, we introduce the matrices $\a_m$ and $\s_m$, and present limit theorems for these matrices. 
Furthermore, we give the asymptotics of 
the averages of $\frac{N_m}{q^{m/2}}$ for Ramanujan graphs.  
Finally, we obtain the asymptotics of the averages of 
the $p^m$th Fourier coefficients of the cusp forms related to  
the LPS Ramanujan graphs.

\section{Preliminaries}
Let $\Z_{\ge a}$ for $a \in \Z$ denote the set of all integers $m$ such that $m \ge a$.
For $x \in \mathbb{R}$, let $[x]$ denote the maximal integer not greater than $x$.

We denote by $\I_n$ and $\O_n$ the $n\times n$ unit matrix and the zero matrix, respectively.
In this paper, we assume that graphs and digraphs are finite. 
Let $G$ be a connected graph and let $E=E(G)$ and $D(G)$ be 
the edge set of $G$ and  the symmetric digraph 
corresponding to $G$, respectively. Here, $D(G)$ is defined as 
$D(G)= \{ (u,v),(v,u) \mid uv \in E\} $. We also identify  $D(G)$ with a graph $G$. For an arc $e=(u,v) \in D(G)$, set its origin  $o(e):=u$ and 
terminus $t(e):=v$, respectively.
Furthermore, 
we denote  the {\em inverse} of $e=(u,v)$ by $e^{-1}=(v,u)$. 

A sequence $P=(e_1, \ldots, e_n )$ of $n$ arcs such that $e_i \in D(G)$,
$t(e_i)=o( e_{i+1} )$  $(1 \leq i \leq n-1)$ is called a {\em path of length $n$} in $D(G)$ (or $G$). 
For a path $P$, we set $|P|:=n$, $o(P):=o(e_1)$ and $t(P):=t(e_n)$. 
We say that a path $P=(e_1, \ldots, e_n )$ has a {\em backtracking} 
if $ e^{-1}_{i+1} =e_i $ for some $i$ $(1 \leq i \leq n-1)$. 
A path $P=(e_1, \ldots, e_n )$ is called a {\it cycle} (or {\it closed path}) if $o(P)=t(P)$. 
The {\em inverse cycle} of a cycle 
$C=( e_1, \ldots, e_n )$ is the cycle 
$C^{-1} =( e^{-1}_n, \ldots, e^{-1}_1 )$. 
For a cycle $C=(e_1, \ldots, e_m )$, let $[C]$ be 
the set of the cyclic arrangements of $C$: 
\[
[C]:=\{
(e_1,\ldots, e_m), \ \  (e_2,\ldots, e_m, e_1),\ \ \dots, \ \  (e_m, e_1,\ldots, e_{m-1})
\}.
\]
For two cycles $C_1$ and $C_2$ in $G$,  $C_1$ is {\it equivalent} to $C_2$ if $[C_1]=[C_2]$. 
Let $B^r$ be the cycle obtained by going $r$ times around a cycle $B$. 
Such a cycle is called a {\em multiple} of $B$. 
A cycle $C$ is called {\em reduced} if 
both $C$ and $C^2 $ have no backtracking,
and a cycle $C=( e_1, \ldots, e_n )$ is said to have a {\it tail} if $e_n=e_1^{-1}$.
Note that a cycle $C$ is reduced if and only if  
$C$ has no backtracking nor tails. 
A cycle $C$ is called {\em prime} if $C\neq B^r$ holds for all 
cycles $B$ with $|B|<|C|$ and integers $r \ge 2$.  
Note that each equivalence class of prime reduced cycles of a graph $G$ 
corresponds to a unique conjugacy class of 
the fundamental group $ \pi {}_1 (G,v)$ of $G$ for a fixed vertex $v$ of $G$. 
Then the Ihara zeta function $Z_G(u)$ of a graph $G$ is 
defined as a function
\[
Z_G (u):= \prod_{[C]} \  (1- u^{ \mid C \mid } )^{-1}
\] 
in $u \in \C$ such that $|u|$ is sufficiently small,
where $[C]$ runs over all equivalence classes of prime reduced cycles 
of $G$ (cf.\ \cite{Ihara,Sunada,Bass,Hashimoto,Stark-Terras96}). 
For a graph $G$ and for $m\in \mathbb{Z}_{\ge 1}$, set
\begin{align}\label{Nm}
N_m
&:=\#\{ C \mid  \textrm{$C$ is a reduced cycle of length $m$ in $G$}\}
\\[3pt]
&=\#\left\{ C \left| 
\begin{array}{l}
\textrm{$C$ is a cycle of length $m$}\\[3pt]
\textrm{without backtracking nor tails in $G$}
\end{array}
\right.
\right\}.
\notag
\end{align}
Note that $Z_G(u)$ is the generating function of $N_m$
as in  \eqref{zeta-exp} \cite[p.137, (2.1)]{Stark-Terras96}.

Let $G$ be a connected graph with $n$ vertices $v_1, \ldots, v_n $, 
and $n \in {\Z_{\ge 1}} $. 
The {\em adjacency matrix} $\A= \A (G)=(a_{ij} )$ is 
the $n\times n$ matrix such that 
\[
a_{ij}
=\left\{
\begin{array}{ll}
\textrm{the number of undirected edges connecting $v_i$ to $v_j$,} 
& \textrm{if $i\neq j$,}
\\[3pt]
2\times \textrm{the number of loops at $v_i$,}
& \textrm{if $i=j$.}
\end{array}
\right.
\]
We write $\D =( d_{ij} )$ for the diagonal matrix 
with $d_{ii} = \deg v_i$ and $d_{ij} =0$ $(i \neq j)$.
Let ${\rm Spec}(\A)$ be the multiset of all eigenvalues of $\A$.

The numbers $\{N_m\}_{m\ge 1}$ are related to 
$\A$ and $\D$ by the following determinant expression for the Ihara zeta function \cite{Ihara, Bass}:
\begin{theorem}[Ihara-Bass]\label{Ihara-Bass} 
Let $G$ be a connected graph with $n$ vertices $v_1 , \ldots , v_n$ and $m$ edges. 
Then the reciprocal of the Ihara zeta function of $G$ is given by 
\[
Z_G(u)^{-1} =(1- u^2 )^{r-1} 
\det ( \I_n -u \A+ u^2 ( \D - \I_n )),  
\]
where $r=m-n+1$ is the first Betti number of $G$.

In particular, if $G$ is a connected $(q+1)$-regular graph with $n$ vertices then 
\begin{align}\label{Ihara-formula}
Z_G(u)^{-1} 
&= (1- u^2 ) {}^{m-n} \det ( \I_n -u \A+q u^2 \I_n ) 
\\
&= (1- u^2 ) {}^{(q-1)n/2}\prod_{\lambda \in {\rm Spec}(\A)} (1- \lambda u+q u^2 ).
\notag  
\end{align}
\end{theorem}

In \cite[Theorem 1.5]{Anantharaman}, Anantharaman obtained a variant of the Ihara-Bass formula for irregular graphs, which has the advantage of involving the characteristic polynomial of $\A$. In \cite[Lemma 2.1]{AnantharamanSabri},
Anantharaman and Sabri discussed certain weights of non-backtracking paths by the Green functions on the universal covering trees
of irregular graphs.
The weights are a generalization of $\M_m$.

Now, let $G$ be a connected $(q+1)$-regular graph. 
Let $h(\theta)$ be an analytic function $\mathbb{\R} \to \mathbb{\C}$
with the following three conditions:
(1) $h( \theta + 2 \pi )=h( \theta )$, (2) $h(- \theta )=h( \theta )$, 
(3) $\sum_{m=1}^{\infty} q^{m/2}|\widehat{h}(m)|<\infty$ (the Ahumada convergence condition).
Here $\widehat{h}$ is the {\em Fourier transform} of $h$ defined by 
\[
\widehat{h} (m)= \frac{1}{ 2 \pi } \int^{ 2 \pi }_0 h( \theta )\, 
e^{ -\sqrt{-1} m \theta }\, d \theta,  \qquad (m \in \mathbb{Z}).
\]
By the condition (3), $h(\theta)=\sum_{m\in\Z}\widehat{h}(m)e^{\sqrt{-1}m\theta}$ is regarded as a function on $-\log\sqrt{q}\le \Im\,\theta \le \log \sqrt{q}$.

The numbers $\{N_m\}_{m\geq 1}$ are related to the eigenvalues of $\A$ by the 
Selberg trace formula of a regular graph \cite{Ahumada,VN}: 
\begin{theorem}[Ahumada]\label{STF}
Let $G$ be a connected $(q+1)$-regular graph with $n$ vertices.
For $\lambda \in \Spec(\A)$,  
we set 
\[
\theta_\lambda=\arccos \left(\frac{\lambda}{2\sqrt{q}}\right) \in \C, 
\]
where $0\le\Re\,\theta_\lambda\le \pi$ and $-\log\sqrt{q}\le \Im\,\theta_\lambda\le \log \sqrt{q}$.  Then we have the trace formula
\begin{align*}
&\sum_{ \lambda \in  \Spec(\A) } h( \theta_\lambda)
\\[3pt]
&= \frac{2nq(q+1)}{ \pi } 
\int^{ \pi }_0 \frac{ \sin^2 \theta }{(q+1 )^2 -4q \cos^2 \theta } \, 
h( \theta )\, d \theta 
+ \sum_{m=1}^{ \infty } N_m q^{-m/2}\,\widehat{h}(m), 
\notag 
\end{align*}
where $N_m$ is the number of reduced cycles of length $m$ in $G$. 
\end{theorem}
For a general theory of the Selberg trace formula, 
refer to \cite{Terras1999, VN}. 

\section{The properties for the matrices on non-backtracking paths of graphs}

Let $G$ be a finite connected $(q+1)$-regular graph with $n$ vertices $ v_1, \ldots, v_n $.  
For an integer $m\ge 1$, we consider an $n \times n$ matrix 
$\A_m$ such that the $ij$-entry of $\A_m$ is the number of 
non-backtracking paths from $v_i$ to $v_j$ of length $m$ in $G$ (see \cite[(2.3)]{Stark-Terras96}).
The matrix $\A_m$ was discussed in \cite{Ihara, LPS, Sarnak, Stark-Terras96, Serre97, DSV}. 
Put $\A = \A (G)$.
Then it holds that 
\[
\A_1 = \A, \quad \A_2 = \A^2 - (q+1)\I_n,  
\]
and $\A_m$ satisfies the following recurrence relation: 
\[
\A_m = \A_{m-1}\, \A - q\A_{m-2} \quad  (m \geq 3). 
\]
For $m=0$, we set $\A_0 = \I_n$. 
Relating to the matrix $\A_m$, we define an $n\times n$ matrix $\M_m$
for $m\in \mathbb{Z}_{\ge 1}$ as
\begin{align}\label{def:M_m}
\M_m := \A_m -(q-1)\sum^{[ \frac{m-1}{2} ]}_{k=1} \A_{m-2k}.
\end{align}
By the same arguments as \cite[\S 2]{Stark-Terras96}, 
the $ii$-entry of $\M_m$ is the number of 
reduced $v_i$-cycles of length $m$ in $G$.  
Thus, the trace $N_m =\mathrm{Tr}(\M_m)$
of $\M_m$ is the number of reduced cycles of length $m$ in $G$
(\cite[(2.7)]{Stark-Terras96}). 

For $m \in \mathbb{Z}_{\ge0}$, let $T_m$ be  the {\em Chebyshev polynomial of the first kind} defined by 
\[
T_m ( \cos \theta )= \cos m \theta. 
\]
Then the matrix $\M_m $ is written in terms of $T_m$.

\begin{prop}\label{thm:Mm-Tm}
Let $G$ be a connected $(q+1)$-regular graph with $n$ vertices.  
Then 
\begin{align}\label{M_m-T_m}
\M_m = 2 q^{m/2} T_m \left( \frac{ \A}{2 \sqrt{q}} \right)+ e_m (q-1) \I_n 
\end{align}
for any $m \in \Z_{\ge 1}$, where $e_m$ is defined by 
\[
e_m :=\left\{
\begin{array}{ll}
1, & \mbox{if $m$ is even, } \\
0, & \mbox{if $m$ is odd.}
\end{array}
\right.
\]
\end{prop}

\begin{proof}
For convenience, we set $\A_m':=\frac{1}{q^{m/2}}\A_m$ and 
$\A':=\A_1'=\frac{1}{q^{1/2}}\A$. 
Then $\A_0'=\I_n$, $\A_2'=(\A_1')^2-\frac{q+1}{q}\I_n$ and 
\[
\A_m'=\A_{m-1}'\, {\A}'-\A_{m-2}'\quad (m\ge 3)
\]
hold, and whence we obtain 
\begin{align}\label{eq:A'-U}
\A_m'
=
U_m\left(\dfrac{\A'}{2}\right)
-
\dfrac{1}{q}U_{m-2}\left(\dfrac{\A'}{2}\right)
\end{align} 
(cf.\ \cite[\S 2.3 and \S 8.2]{Serre97} and \cite[p.6, Remark]{HSSS}). Here $U_m(x)$ is the {\em Chebyshev polynomial of the second kind} defined by 
\[
U_m ( \cos\theta )= \dfrac{\sin (m+1) \theta}{\sin \theta}, \quad (m \in \Z_{\ge 0}),
\]
and we set $U_{-1}(x)=0$. Then, the family $\{U_m(x)\}_{m\ge -1}$ satisfies the recurrence relation $U_{m+2}(x)=2xU_{m+1}(x)-U_m(x)$ for $m \in \Z_{\ge -1}$. Now let us take any $m \in \Z_{\ge1}$.
By the definition \eqref{def:M_m} of $\M_m$, 
we obtain
\begin{align*}
\dfrac{\M_m}{q^{m/2}} = \A_m' 
-(q-1)\sum^{[ \frac{m-1}{2} ]}_{k=1} q^{-k}\A_{m-2k}'.
\end{align*} 
Hence, by \eqref{eq:A'-U}, $q^{-m/2}\M_m$ is expressed as
\begin{align}\label{eq:A'-U-1}
\dfrac{\M_m}{q^{m/2}}
=& 
U_m\left(\dfrac{\A'}{2}\right)
-
\dfrac{1}{q}U_{m-2}\left(\dfrac{\A'}{2}\right)
\\[3pt]
&-(q-1)\sum^{[ \frac{m-1}{2} ]}_{k=1} q^{-k}
\left\{
U_{m-2k}\left(\dfrac{\A'}{2}\right)
-
\dfrac{1}{q}U_{m-2(k+1)}\left(\dfrac{\A'}{2}\right)
\right\}.
\notag
\end{align}
The sum in the right-hand side above can be written as
\begin{align}
\label{eq:A'-U-2}
&(q-1)\sum^{[ \frac{m-1}{2} ]}_{k=1} q^{-k}
\left\{
U_{m-2k}\left(\dfrac{\A'}{2}\right)
-
\dfrac{1}{q}U_{m-2(k+1)}\left(\dfrac{\A'}{2}\right)
\right\}
\\[3pt]
&=
\left(
\sum_{k=0}^{[ \frac{m-1}{2} ]-1}
-\sum_{k=1}^{[ \frac{m-1}{2} ]}
\right)q^{-k}(1-q^{-1})U_{m-2(k+1)}\left(\dfrac{\A'}{2}\right)
\notag \\[3pt]
&=(1-q^{-1})U_{m-2}\left(\dfrac{\A'}{2}\right)
-q^{-[ \frac{m-1}{2} ]}(1-q^{-1})U_{m-2-2[ \frac{m-1}{2} ]}\left(\dfrac{\A'}{2}\right).
\notag 
\end{align}
The second term in the equation above is evaluated as 
\begin{align}\label{eq:A'-U-3}
&q^{-[ \frac{m-1}{2} ]}(1-q^{-1})U_{m-2-2[ \frac{m-1}{2} ]}\left(\dfrac{\A'}{2}\right)
\\[3pt]
&=\left\{
\begin{array}{ll}
q^{-m/2}(q-1)U_0(\A'/2),& \mbox{if $m$ is even, } \\
q^{-(m-1)/2}(1-q^{-1})U_{-1}(\A'/2), & \mbox{if $m$ is odd,}
\end{array}
\right.
\notag \\[3pt]
&=q^{-m/2}e_m(q-1)\I_n
\notag
\end{align}
with the aid of $U_0(x)=1$, $U_{-1}(x)=0$ and the definition of $e_m$.
Combining \eqref{eq:A'-U-1}, \eqref{eq:A'-U-2}, \eqref{eq:A'-U-3} and 
the identity $2\,T_m(x)=U_{m}(x)-U_{m-2}(x)$,
the assertion follows.
\end{proof}
Taking the traces of the both sides of \eqref{M_m-T_m}, 
it follows that
\begin{align}\label{tr:M_m-T_m}
N_m=  2 q^{m/2} \sum_{\lambda \in \Spec(\A)} T_m\left(\frac{\lambda}{2\sqrt{q}}\right)+ne_m(q-1).
\end{align}
This equation \eqref{tr:M_m-T_m} and  equations equivalent to \eqref{tr:M_m-T_m} have already obtained in many previous works (for example, see \cite[(16), p.229]{Ihara}, \cite[Lemme 3]{Serre97}, \cite[Lemma 4]{Rangarajan} and \cite[(8)]{Huang}\footnote{Huang's $T_k(x)$ in \cite[(8)]{Huang} is not the Chebyshev polynomial of the first kind, and his $T_k(x)$ is equal to our $2\,T_k(x/2)$ as in \cite[\S 3]{Huang}.}).

\begin{remark} \label{LPS}
Similarly to $\M_m$, the matrix $\T_m$ concerned with certain cycles
was considered in the context of Ramanujan graphs.
	For example, 
	let $G=X^{p,q}$ be the LPS Ramanujan graph defined in \cite{LPS}. 
	Here $p$ and $q$ are prime numbers such that $p\equiv q \equiv 1 \pmod{4}$ and $p\neq q$. 
	For such $p$ and $q$, $G$ is constructed as a $(p+1)$-regular Cayley graph, and the number of vertices $n$ is equal to $\#{\rm PGL}_2(\F_q)=q(q^2-1)$ and $\# {\rm PSL}_2(\F_q)=\frac{q(q^2-1)}{2}$ according to $(\frac{p}{q})=-1$ and $(\frac{p}{q})=1$, respectively. Here $(\frac{\cdot}{\cdot})$ is the Legendre symbol. Set 
	\[
	\T_m :=\sum_{0\le r \le m/2}\A_{m-2r}=p^{m/2} U_{m}\left(\dfrac{\A}{2\sqrt{p}}\right)
	\]
	{\rm (}cf.\ \cite[p.23]{DSV}{\rm )}.
	Then the trace of $\T_m$ is written as
	\begin{align*}
	\dfrac{2}{n}\Tr(\T_m)
	&=\dfrac{2}{n} p^{m/2} \sum_{\lambda \in {\rm Spec}(\A) } U_{m}\left(\dfrac{\lambda}{2\sqrt{p}}\right)
	=
	C(p^m)+a(p^m), 
	\end{align*}
	where $C(p^m)$ and $a(p^m)$ are the $p^m$th Fourier coefficients of an Eisenstein series and
	of a cusp form of weight $2$ on $\Gamma(16q^2)$, respectively.
	Indeed, the trace of $\frac{2}{n}\Tr(\T_m)$ is the $p^m$th Fourier coefficient of the theta function $\Theta(z)=\sum_{x \in \Z^4}e^{2\pi i Q(x)z}$
	associated with the quadratic form $Q(x_1, x_2, x_3, x_4)=x_1^2+4q^2x_2^2+4q^2x_3^2+4q^2x_4^2$ {\rm (}cf.\ \cite[p.272]{LPS}{\rm)}.
	Let $f_{m,v}$ be the number of non-backtracking cycles {\rm(}which may have tails{\rm)} from a vertex $v$ to $v$ of length $m$ in $G$.
Note that $f_{m,v}=(\A_m)_{v,v}$, that is, $f_{m,v}$ is the $vv$-entry of the matrix $\A_m$.  
 Since $G$ is vertex-transitive, $f_{m,v}$ is constant for all $v\in V(G)$. Set $f_m=f_{m,v}$. Then, we obtain
	\[
	\Tr(\T_m)=n\sum_{0\le r \le m/2} f_{m-2r}.
	\] 
	Thus it holds that 
	\[
	2\sum_{0\le r \le m/2} f_{m-2r} = C(p^m)+a(p^m).
	\]
	Here $C(p^m)$ is explicitly calculated as 
	$$C(p^m) = \frac{1+(\frac{p}{q})^m}{2}\frac{4}{q(q^2-1)}
	\frac{p^{m+1}-1}{p-1},$$
	by \cite[(4.19) and (4.20)]{LPS}, and
Deligne's bound	$a(p^m) = O_\epsilon(p^{m(1/2+\epsilon)})$ holds for any $\epsilon>0$
{\rm (}cf.\ \cite{Deligne1}, \cite{Deligne2}, \cite[Theorem 4.5.17]{Miyake}{\rm )}.
\end{remark} 

\section{The limit theorem with respect to the matrices on non-backtracking paths  
of a regular graph} 

Let $G$ be a connected $(q+1)$-regular graph with $n$ vertices and let 
$\A=\A(G)$ be its adjacency matrix. 
Furthermore, let $\sigma(\A)$ be the set of distinct 
eigenvalues of $\A$.  For $\lambda \in \sigma(\A)$, let $\P_\lambda$ be the projection into the eigenspace of $\lambda$. 
Then we have 
\[
\A= \sum_{\lambda \in \sigma(\A)}\lambda \P_\lambda. 
\]
Note that 
\[
\textrm{$\P^2_\lambda =\P_\lambda$  \ and \  $\P_\lambda \P_\mu = \O_n$ \quad ($\lambda, \mu \in \sigma(\A)$, $\lambda \neq \mu$). }
\]
Furthermore, we have 
\[
\P_{q+1} = \frac{1}{ \sqrt{n}}  
\left[
\begin{array}{c}
1 \\
\vdots \\
1 
\end{array}
\right]
\dfrac{1}{\sqrt{n}}  
\left[
\begin{array}{ccc}
1 & \cdots & 1 
\end{array}
\right]
= 
\frac{1}{n} 
\left[
\begin{array}{ccc}
1 & \cdots & 1 \\ 
\vdots & \ddots  & \vdots \\
1 & \cdots & 1 
\end{array}
\right] 
= \frac{1}{n} {\bf J}_n,  
\]
where ${\bf J}_n $ is the $n \times n$ matrix with all entries being one. 
Moreover, if $X$ is bipartite then $-(q+1)\in \sigma(\A)$ and we can take an eigenvector  
\[
\dfrac{1}{\sqrt{n}}{}^t[\underbrace{1,\dots,1}_{n/2},\underbrace{-1,\dots,-1}_{n/2}], 
\]
with respect to the eigenvalue $-(q+1)$ under a suitable labeling of vertices. In this case, we have
\[
\P_{-(q+1)}=\dfrac{1}{n}
\left[\begin{array}{cc}
 \J_{n/2} & -\J_{n/2} \\
 -\J_{n/2} & \J_{n/2}
 \end{array}\right].
\]
Note that $\Tr(\P_{q+1})=\Tr(\P_{-(q+1)})=1$.

Thus, we have 
\[
\A^m = \sum_{\lambda\in \sigma(\A)}\lambda^m \P_\lambda. 
\]
By Proposition \ref{thm:Mm-Tm}, we have 
\[
\M_m = 2 q^{m/2} \sum_{\lambda\in \sigma(\A)} T_m \left( \frac{ \lambda}{2 \sqrt{q}} \right) \P_\lambda + e_m (q-1) \I_n. 
\]

For $m\in \mathbb{Z}_{\ge 1}$, we define a new matrix $\a_m $ as follows: 
\begin{align}\label{a_m}
\a_m 
&:= 
\dfrac{1}{2 q^{m/2}}
\left\{ \M_m - 2 q^{m/2} \sum_{\lambda \in \sigma(\A) \atop |\lambda|\geq 2\sqrt{q}} T_m \left( \frac{ \lambda}{2 \sqrt{q}} \right) \P_\lambda - e_m (q-1) \I_n \right\} 
\\[6pt]
&= \sum_{\lambda \in \sigma(\A) \atop | \lambda |< 2 \sqrt{q}} T_m \left( \frac{\lambda }{2 \sqrt{q}} \right) \P_{\lambda }.   
\notag
\end{align}
Let $\B$ be an $n\times n$ matrix $\B=f(\A)\in \mathbb{R}[\A]$, where $f(x)$  is a polynomial in $x$. 
For $\B$, we call the matrices 
$
\sum_{\lambda \in \sigma(\A) \atop | \lambda |< 2 \sqrt{q}} f(\lambda)\P_\lambda
$
and 
$
\sum_{\lambda \in \sigma(\A) \atop | \lambda |\ge 2 \sqrt{q}} f(\lambda)\P_\lambda
$ 
the {\it principal part} 
and 
the {\it singular part} 
of $\B$ with respect to $\A$, respectively. 
Then $\a_m$ is the  principal part of $\frac{1}{2q^{m/2}}\{\M_m-e_m(q-1)\I_n\}$ with respect to $\A$.

For example, if $G$ is a Ramanujan graph such that $2\sqrt{q} \not \in \sigma(\A)$, then 
\begin{align}\label{a_m-Ramanujan}
\a_m=\dfrac{1}{2 q^{m/2}}\left\{ \M_m - 2 q^{m/2} \sum_{\lambda \in \sigma(\A) \atop \lambda\in\{\pm(q+1), -2\sqrt{q}\}} T_m \left( \frac{ \lambda}{2 \sqrt{q}} \right) \P_\lambda - e_m (q-1) \I_n \right\}.
\end{align}
The range of the sequence $\a_m$ is given in  Proposition \ref{range} below.

\begin{prop}\label{range}
Let $G$ be a connected $(q+1)$-regular graph with $n$ vertices $v_1,\dots, v_n$. 
For $m\in \mathbb{Z}_{\ge 1}$ and $1\le i, j \le n$, every $ij$-entry $(\a_m)_{ij}$ of $\a_m$ satisfies the following inequality{\rm :}
\begin{align*}
-1 \le (\a_m)_{ij} \le 1.
\end{align*}
\end{prop}
\begin{proof}
Let $V$ be the set of vertices of $G$ and $C(V)$ the space of all real-valued functions on $V$. By fixing an enumeration $V=\{v_1,\ldots,v_n\}$, we identify $C(V)$ with $\R^n$
so that the canonical basis $\{e_1,\ldots, e_n\}$ of $\R^n$ correspond to
the basis $\{\psi_1,\ldots, \psi_n\}$ of $C(V)$, where $\psi_j$ is the characteristic function of the singleton $\{v_j\} \subset V$.
Let $\{\phi_{\lambda, 1},\dots,  \phi_{\lambda, m_\lambda}  \}$ be an orthonormal basis of 
the eigenspace $W_\lambda$ of an eigenvalue $\lambda\in \sigma(\A)$.  
Here $m_\lambda$ denotes the multiplicity of an eigenvalue $\lambda$ of $\A$. 
Since $\A$ is symmetric, 
$
\bigcup_{\lambda \in  \sigma(\A)}\{\phi_{\lambda, 1},\dots,  \phi_{\lambda, m_\lambda}  \}
$
becomes an orthonormal basis of $C(V)$.
Then the  $ij$-entry of 
the spectral projection $\P_\lambda$ is written as
\[
(\P_\lambda)_{ij}
=\sum_{\ell=1}^{m_\lambda} \phi_{\lambda,\ell} (v_i)\, \phi_{\lambda,\ell} (v_j).
\]
For $\lambda \in\sigma(\A)$ with $|\lambda|< 2\sqrt{q}$, set 
$\theta_\lambda = \arccos (\frac{\lambda}{2 \sqrt{q}})$ ($0< \theta_\lambda < \pi$). 
Then  the  $ij$-entry of $\a_m$ is written as
\begin{align}\label{am-ij}
(\a_m)_{ij}
=
\sum_{\lambda \in \sigma(\A) \atop |\lambda|< 2\sqrt{q}}
\cos m\theta_\lambda
\sum_{\ell=1}^{m_\lambda}
\phi_{\lambda,\ell} (v_i)\, \phi_{\lambda,\ell} (v_j).
\end{align}
Since $(\phi_{\lambda,1},\ldots, \phi_{\lambda, m_\lambda})_{\lambda \in \sigma(\A)}$ is regarded as an orthogonal matrix, we have
\[
\sum_{\lambda \in \sigma(\A)}\sum_{\ell=1}^{m_\lambda} \phi_{\lambda,\ell}(v)^2=1.
\]
By the Cauchy-Schwarz inequality, we have
\begin{align*}
\sum_{\lambda \in \sigma(\A) \atop |\lambda|< 2\sqrt{q}}
\sum_{\ell=1}^{m_\lambda}|\phi_{\lambda,\ell} (v_i)\, \phi_{\lambda,\ell} (v_j)
|
\le
\sqrt{
(
\sum_{\lambda \in \sigma(\A) \atop |\lambda|< 2\sqrt{q}}
\sum_{\ell=1}^{m_\lambda}
 |\phi_{\lambda,\ell}(v_i)|^2
) 
(
\sum_{\lambda \in \sigma(\A) \atop |\lambda|< 2\sqrt{q}}
\sum_{\ell=1}^{m_\lambda}
 |\phi_{\lambda,\ell}(v_j)|^2
)
} 
\le 1. 
\end{align*}
From this and \eqref{am-ij}, we are done. 
\end{proof}
Then the following limit theorem on $\a_m $ holds. 
\begin{theorem}\label{thm:moment}
Let $G$ be a connected $(q+1)$-regular graph with $n$ vertices, and 
let $k$ be a positive integer. 
We have
\begin{align}\label{moment-1}
	\bigg\| \frac{ \a^k_1 + \a^k_2+ \cdots + \a^k_N }{N}
	&-e_k\dfrac{1}{2^k } \binom{k}{k/2}\sum_{\substack{\lambda \in \sigma(\A) \\ |\lambda| < 2\sqrt{q}}}  \P_\lambda \\
&	-\frac{1}{2^{k-1}}\sum_{j=0}^{[(k-1)/2]}
\binom{k}{j}\sum_{\substack{\lambda\in \sigma(\A) \\ |\lambda|<2\sqrt{q}\\ \theta_\lambda \in (k-2j)^{-1}2\pi \mathbb{Z}}}\P_\lambda \bigg\|
	= O\left(\frac{1}{N}\right), \notag 
\end{align}
where $e_m$ is the same as in Proposition \ref{thm:Mm-Tm}.
In particular, if we assume
$\theta_\lambda \not\in \bigcup_{j=0}^{[(k-1)/2]}(k-2j)^{-1}2\pi \mathbb{Z}$
for all $\lambda \in \sigma(\A)$ with $|\lambda|<2\sqrt{q}$,
then we have
\begin{align}\label{moment}
	\left\| \frac{ \a^k_1 + \a^k_2+ \cdots + \a^k_N }{N} -
	e_k\dfrac{1}{2^k } \binom{k}{k/2}\sum_{\substack{\lambda \in \sigma(\A) \\ |\lambda| < 2\sqrt{q}}}  \P_\lambda \right\|=O\left(\frac{1}{N}\right)
\end{align}
as $N\to \infty$,
where $\| \cdot \|$ is the Euclidean norm on $\mathbb{R}^{n^2}$.

\end{theorem}
\begin{proof}
We have 
\begin{align*}
\a^k_m = \sum_{\lambda \in \sigma(\A)\atop |\lambda|<2\sqrt{q}} \cos^k m\theta_\lambda\P_\lambda. 
\end{align*}
From this, it holds that
\begin{align*}
\frac{1}{N}\sum_{m=1}^N \a^k_m
=
\frac{1}{N}\sum_{m=1}^N
\sum_{\lambda \in \sigma(\A) \atop |\lambda|<2\sqrt{q}}
\cos^k m\theta_\lambda \P_\lambda. 
\end{align*}
Taking into account for
\[
\cos^k \theta = 
\left\{
\begin{array}{ll} 
\frac{1}{2^k } \binom{k}{k/2}+ \frac{2}{2^k} \sum^{\frac{k-2}{2}}_{j=0} \binom{k}{j}\cos ((k-2j) \theta ),  
& \mbox{if $k$ is even, }\\
\frac{2}{2^k } \sum^{\frac{k-1}{2}}_{j=0} \binom{k}{j}\cos ((k-2j) \theta ), & \mbox{if $k$ is odd,}
\end{array}
\right. 
\]
we discuss two cases.  

If $k$ is even, then we have 
\begin{align*}
&\frac{1}{N}\sum_{m=1}^N \a_m^k 
=
\frac{1}{N} \sum^N_{m=1} \sum_{\lambda \in \sigma(\A) \atop |\lambda|<2\sqrt{q}}
\left\{\dfrac{1}{2^k}\binom{k}{k/2}+ \frac{2}{2^k} \sum^{\frac{k-2}{2}}_{j=0}\binom{k}{j} \cos ((k-2j) m\theta_\lambda ) \right\} \P_\lambda 
\\
&=
\dfrac{1}{2^k}\binom{k}{k/2}\sum_{\lambda \in \sigma(\A) \atop |\lambda|<2\sqrt{q}}\P_\lambda 
+ \frac{2}{2^k} \sum^{\frac{k-2}{2}}_{j=0}\binom{k}{j}  \sum_{\lambda \in \sigma(\A) \atop |\lambda|<2\sqrt{q}}
\dfrac{\sum^N_{m=1}\cos ((k-2j) m\theta_\lambda ) }{N}\P_\lambda.
\end{align*}
Note the formula
\[
\sum^N_{m=1}\cos m\varphi 
= 
\begin{cases}
N, & \text{if   $\varphi \in 2\pi \mathbb{Z}$,}
\\
\dfrac{\sin(N+\frac{1}{2})\varphi}{2\sin \frac{\varphi}{2}}-\dfrac{1}{2}=O(1), & \text{if   $\varphi \in \mathbb{R}-2\pi \mathbb{Z}$}
\end{cases}
\]
as $N\rightarrow \infty$. For an angle $\theta_\lambda$ such that $\theta_\lambda \not\in \bigcup_{j=0}^{[(k-1)/2]}(k-2j)^{-1}2\pi \mathbb{Z}$, 
it follows that
$
\sum^N_{m=1}\cos ((k-2j) m\theta_\lambda)=O(1).
$
Thus we have the assertion when $k$ is even.
The case where $k$ is odd is proved similarly.
\end{proof}

We remark that in general the implied constants of \eqref{moment-1} and of \eqref{moment} depend on $G$ (especialy on $n$).  
Such a dependence is applied to Corollary \ref{average-Nm}, Theorem \ref{thm:moment-sm} and Corollary \ref{cor:LPS} below.
We also remark that the assumption $\theta_\lambda \not\in \bigcup_{j=0}^{[(k-1)/2]}(k-2j)^{-1}2\pi \mathbb{Z}$ on the eigenvalues $\lambda$ is complicated.
For example, this assumption is equivalent to ``$2\sqrt{q} \notin \sigma(\A)$'' for $k=1$,
``$\pm 2\sqrt{q} \notin \sigma(\A)$'' for $k=2$, and ``$-\sqrt{q}, 2\sqrt{q} \notin \sigma(\A)$'' for $k=3$, respectively.
It is difficult to find the existence of graphs satisfying these assumptions for all $k\ge 1$.

The asymptotics \eqref{moment} closely resembles the $k$th moment of the arcsine law
($\frac{1}{\pi\sqrt{1-x^2}}dx$ on $[-1,1]$). 
However, $\a_m$ is not a scalar.
Thus we should not conclude that 
the distribution of $\a_m$ is the arcsine law. 
It is a future work to discuss the ``distribution'' of $\a_m$.  

As a corollary of Theorem \ref{thm:moment}, we exhibit an asymptotic 
formula of the average of $\frac{N_m}{q^{m/2}}$ for a Ramanujan graph. 
\begin{corollary}\label{average-Nm}
For a natural number $q>1$, let $G$ be a connected $(q+1)$-regular Ramanujan graph with $n$ vertices.
Let $N_m$ be the 
number of reduced cycles of length $m$ in $G$. 
Then, as $N\to\infty$, we have
\begin{align*}
\dfrac{1}{N}\sum_{m=1}^N \frac{N_m}{q^{m/2}}
 & = 
\left\{
\begin{array}{ll}
\dfrac{1}{N}\dfrac{2q^{[N/2]+1}}{q-1}+2m_{2\sqrt{q}}
+O\left(\dfrac{1}{N}\right), & \textrm{if $G$ is bipartite},\\[12pt]
\dfrac{1}{N}\dfrac{q^{(N+1)/2}}{q^{1/2}-1}+2m_{2\sqrt{q}}
+O\left(\dfrac{1}{N}\right), & \textrm{if $G$ is non-bipartite},
\end{array}
\right.
\end{align*}
where $m_{2\sqrt{q}}\in\Z_{\ge1}$ is the multiplicity of $2\sqrt{q}$ if $2\sqrt{q} \in \sigma(\A)$ and $m_{2\sqrt{q}}=0$ otherwise.
\end{corollary}
\begin{proof}
By \eqref{a_m-Ramanujan} and $\Tr(\P_\lambda)=m_\lambda$, we obtain
\[
\Tr(\a_m)
=\dfrac{N_m-e_m(q-1)n}{2q^{m/2}}
-\sum_{\lambda \in \sigma(\A) \cap\{\pm(q+1),\pm2\sqrt{q}\}}m_{\lambda}T_m\left(\dfrac{\lambda}{2\sqrt{q}}\right),
\]
which gives us
\begin{align*}
\dfrac{2}{N}\sum_{m=1}^N\Tr(\a_m)
= &
\dfrac{1}{N}\sum_{m=1}^N\dfrac{N_m}{q^{m/2}}
-\dfrac{1}{N}\sum_{m=1}^N\dfrac{e_m(q-1)n}{q^{m/2}} \\
& -\dfrac{1}{N}\sum_{m=1}^N
\sum_{\lambda=\pm(q+1)}2\,m_\lambda \,T_m\left(\dfrac{\lambda}{2\sqrt{q}}\right)\\
&
-\dfrac{1}{N}\sum_{m=1}^N \sum_{\lambda=\pm2\sqrt{q}}2\,m_{\lambda}\,T_m\left(\dfrac{\lambda}{2\sqrt{q}}\right).
\end{align*}
Here we put $m_\lambda=0$ if $\lambda \notin\sigma(\A)$. Note
$\sum_{m=1}^N\tfrac{e_m(q-1)n}{q^{m/2}}
=O(1)$ and
$$\sum_{m=1}^{N}2\,m_{-2\sqrt{q}}\,T_m(-1) = 2\,m_{-2\sqrt{q}}\sum_{m=1}^{N} (-1)^m = O(1),$$
$$\sum_{m=1}^{N}2\,m_{2\sqrt{q}}\,T_m(1) = \sum_{m=1}^{N}2\,m_{2\sqrt{q}}=2\,m_{2\sqrt{q}}N.$$
The left-hand side of the equality above is estimated as $O(\frac{1}{N})$ by 
Theorem \ref{thm:moment} for $k=1$.
By $2\,T_{m}(\frac{\pm (q+1)}{2\sqrt{q}}) = (\pm q^{1/2})^{m}+(\pm q^{1/2})^{-m}$,
we obtain
\begin{align*}\dfrac{1}{N}
\sum_{m=1}^{N}2\,T_m\left(\dfrac{\epsilon(q+1)}{2\sqrt{q}}\right)
= & \frac{1}{N}\left( \frac{\epsilon q^{1/2}-\epsilon^{N+1}q^{(N+1)/2}}{1-\epsilon q^{1/2}} + 
\frac{\epsilon q^{-1/2}-\epsilon^{N+1}q^{-(N+1)/2}}{1-\epsilon q^{-1/2}}\right)\\
= & \frac{1}{N}\frac{\epsilon^{N+1}q^{(N+1)/2}}{\epsilon q^{1/2}-1} + O\left(\frac{1}{N}\right)
\end{align*}
for $\epsilon=\pm1$.
By noting that the both terms for $\epsilon=\pm1$ appear if $G$ is bipartite,
and that the term only for $\epsilon=1$ appears if $G$ is not bipartite,
we obtain the assertion.
\end{proof}

Inspired by Remark \ref{LPS} and \eqref{a_m}, we consider $\s_m$ below.
\begin{theorem}\label{thm:moment-sm}
Let $G$ be a connected $(q+1)$-regular graph with $n$ vertices.
For $m\ge 0$, set
\begin{align*}
&\s_m:= \sum_{\lambda \in \sigma(\A) \atop |\lambda| <2\sqrt{q}}
U_m\left(\frac{\lambda}{2\sqrt{q}}\right)\P_\lambda
= \sum_{\lambda \in \sigma(\A) \atop |\lambda| <2\sqrt{q}}
\frac{\sin ((m+1)\theta_\lambda)}{\sin \theta_\lambda}\P_\lambda.
\end{align*}
Then, for any positive integer $k$,
we have 
\begin{align*}
& \bigg\|
\dfrac{1}{N}\sum_{m=1}^N \s_m^k
-
e_k\dfrac{1}{2^k } \binom{k}{k/2}\sum_{\substack{\lambda \in \sigma(\A) \\ |\lambda| < 2\sqrt{q}}}  \dfrac{1}{\sin^k \theta_\lambda} \P_\lambda \\
& -e_k\frac{1}{2^{k-1}}
\sum_{j=0}^{\frac{k-2}{2}}(-1)^{\frac{k}{2}-j}\binom{k}{j}\sum_{\substack{\lambda \in \sigma(\A) \\ |\lambda| < 2\sqrt{q}\\
\theta_\lambda \in  (k-2j)^{-1}2\pi \mathbb{Z}}} \frac{1}{\sin^k \theta_{\lambda}}\P_\lambda
\bigg\|
=O\left(\dfrac{1}{N}\right), \quad N\rightarrow \infty.
\end{align*}
In particular, 
if $\theta_\lambda \notin \bigcup_{j=0}^{[(k-1)/2]}(k-2j)^{-1}2\pi \mathbb{Z}$ for all
$\lambda \in \sigma(\A)$ with $|\lambda|<2\sqrt{q}$, then we have
\begin{align*}
	\bigg\|
	\dfrac{1}{N}\sum_{m=1}^N \s_m^k
	-
	e_k\dfrac{1}{2^k } \binom{k}{k/2}\sum_{\substack{\lambda \in \sigma(\A) \\ |\lambda| < 2\sqrt{q}}}  \dfrac{1}{\sin^k \theta_\lambda} \P_\lambda
	\bigg\|
	=O\left(\dfrac{1}{N}\right), \quad N\rightarrow \infty.
\end{align*}
\end{theorem}
\begin{proof}We start the proof from the expression
\begin{align*}
\frac{1}{N}\sum_{m=1}^{N} \s_m^k
= & \sum_{{\lambda \in \sigma(\A) \atop |\lambda| <2\sqrt{q}}}
\left(\frac{1}{N}\sum_{m=1}^{N}\sin^k((m+1)\theta_\lambda)\right)\frac{1}{\sin^k\theta_\lambda}\P_\lambda.
\end{align*}
By the formula
\[
\sin^k \varphi
=\dfrac{e_k}{2^k } \binom{k}{k/2}
+
\dfrac{2}{2^k}
\sum_{j=0}^{\frac{k-1-e_k}{2}}(-1)^{\frac{k-1+e_k}{2}-j}\binom{k}{j}\cs_k((k-2j)\varphi)
\]
with $\cs_k\, \varphi := e_k\cos \varphi+(1-e_k)\sin \varphi$,
we evaluate $\sum_{m=1}^{N}\sin^k((m+1)\theta_\lambda)$ as
\begin{align*} \frac{Ne_k}{2^k}\binom{k}{k/2}+\frac{2}{2^k}\sum_{j=0}^{\frac{k-1-e_k}{2}}(-1)^{\frac{k-1+e_k}{2}-j}\binom{k}{j}\sum_{m=1}^{N}\cs_k((k-2j)(m+1)\theta_\lambda).
\end{align*}
Note the formula
\begin{align*}
&\sum_{m=1}^{N}\cs_k((m+1)\varphi)
\\
&= 
\begin{cases}
N e_k, & \text{if   $\varphi \in 2\pi \mathbb{Z}$,}
\\
\dfrac{(-1)^k}{2}\dfrac{\cs_{k+1}((N+3/2)\varphi) - \cs_{k+1}(3\varphi/2)}{\sin (\varphi/2)}=O(1), & \text{if   $\varphi \in \mathbb{R}-2\pi \mathbb{Z}$}
\end{cases}
\end{align*}
as $N\rightarrow \infty$. 
For an angle $\theta_\lambda$ such that $\theta_\lambda \not\in \bigcup_{j=0}^{[(k-1)/2]}(k-2j)^{-1}2\pi \mathbb{Z}$, 
it follows that
$
\sum^N_{m=1}\cs_k ((k-2j) (m+1)\theta_\lambda)=O(1).
$ 
Hence we are done.
\end{proof}
From this, we can estimate the average of Fourier coefficients of the cusp form
related with the LPS Ramanujan graph $X^{p,q}$ as follows.
\begin{corollary}\label{cor:LPS}
Let $\A$ be the adjacency matrix of the LPS Ramanujan graph $X^{p,q}$. 
Let $a(p^m)$ be the $p^m$th Fourier coefficient of the cusp form of weight $2$ on $\Gamma(16q^2)$ related to $X^{p,q}$. 
Then we have 
\begin{align*}
\dfrac{1}{N}\sum_{m=1}^N\dfrac{a(p^m)}{2p^{m/2}}
=O\left(\dfrac{1}{N}\right), \qquad N\rightarrow  \infty.
\end{align*}
\end{corollary}
\begin{proof}
Suppose $\pm2\sqrt{p}\notin \sigma(\A)$, which will be checked later.
Let $n$ be the number of vertices of $X^{p,q}$.
Then, we obtain
\begin{align*}
\dfrac{1}{n}\Tr(\s_m) = \frac{1}{n}\sum_{\lambda \in \sigma(\A) \atop |\lambda| <2\sqrt{p}}
\frac{\sin (m+1)\theta_\lambda}{\sin \theta_\lambda}m_\lambda = \dfrac{a(p^m)}{2p^{m/2}}
\end{align*}
by Remark \ref{LPS} (cf.\ \cite[(4.15) and p.273]{LPS}).
Hence, under the assumption $\pm2\sqrt{p}\notin \sigma(\A)$, we complete the proof by Theorem \ref{thm:moment-sm} for $k=1$.

From the argument above, we have only to show $\pm2\sqrt{p} \notin \sigma(\A)$.
Let $T_p$ be the Hecke operator at $p$ on the space of elliptic cusp forms of weight $2$ and level $16q^2$.
%
By \cite[Theorem 4.1]{Coleman}, all eigenvalues $\lambda$ of $T_p$ satisfy the strict inequality
$|\lambda| < 2 \sqrt{p}$. 
Here we use the well known fact that $\lambda$
 is the $p$th Fourier coefficient of a normalized new eigenform of weight $2$
 and level $16q^2$
by \cite[Proposition 6.23]{Iwaniec}.
Since the adjacency matrix $\A$ of $X^{p,q}$ corresponds to the Hecke operator $T_p$
as in the last paragraph of \cite[\S7.2]{Costache}, we obtain $|\lambda|<2\sqrt{p}$ for all $\lambda \in \sigma(\A)-\{\pm(p+1)\}$.
\end{proof}

We can express the generating function of Fourier coefficient $\{a(p^m)\}_{m}$ in terms of 
the Ihara zeta function of $X^{p,q}$.
Note that the Ihara zeta function $Z_{G}(u)$ of a $(q+1)$-regular graph $G$ has
a meromorphic continuation to $\C$ by Ihara's formula \eqref{Ihara-formula}.
Set $Z'_G(u)=\frac{d}{du}Z_G(u)$.
\begin{prop}\label{generating}
Under the same condition as in Corollary \ref{cor:LPS}, we have 
$$\varphi(t):=\sum_{m=1}^{\infty}\frac{a(p^m)}{2p^{m/2}}t^m = \frac{t}{n(1-t^2)}\left\{lt+\frac{1}{\sqrt{p}}\frac{Z_{X^{p,q}}'}{Z_{X^{p,q}}}\left(\frac{t}{\sqrt{p}}\right)
-\frac{(p-1)nt}{p-t^2}+F(t)\right\}$$
and the radius of convergence of the left-hand side is $1$.
Here $l$ is the number of eigenvalues $\lambda$ of $\A$ such that $|\lambda|<2\sqrt{p}$, 
counted with multiplicity, i.e., 
$$
l=\begin{cases} n-2, & \text{if $X^{p,q}$ is bipartite},\\
n-1, & \text{if $X^{p,q}$ is non-bipartite},\end{cases}
$$
and
$$F(t) := \begin{cases}
\frac{-2pt}{1-pt^2}+\frac{-2p^{-1}t}{1-p^{-1}t^2}, & \text{if $X^{p,q}$ is bipartite},\\
\frac{-\sqrt{p}}{1-\sqrt{p}t}+\frac{-\sqrt{p}^{-1}}{1-\sqrt{p}^{-1}t}, & \text{if $X^{p,q}$ is non-bipartite}.
\end{cases}$$
Moreover, we have
$\lim_{t\rightarrow 1}(t-1)\varphi(t)=0.$
%
\end{prop}
\begin{proof}
	By the relation $U_{m}(\frac{\alpha+\alpha^{-1}}{2})=\frac{\alpha^{m+1}-\alpha^{-m-1}}{\alpha-\alpha^{-1}}$ with
$\alpha=e^{\sqrt{-1}\theta}$ and $0<\theta<\pi$, the radius of convergence of $\varphi(t)$ is $1$.
	By the equation $\sum_{m = 0}^{\infty}U_m(x)t^m=\dfrac{1}{1-2xt+t^2}$,
	we have 
	$$
	\sum_{m = 1}^{\infty}U_m(x)t^{m-1}=
	\dfrac{t}{1-t^2}-\dfrac{1}{1-t^2}
	\{\log (1-2xt+t^2)\}',
	$$
	and hence
	\begin{align*}
		&n \varphi(t)= nt\sum_{m=1}^\infty \dfrac{a(p^m)}{2p^{m/2}}t^{m-1}
		=t\sum_{m= 1}^{\infty}\sum_{\lambda \in \Spec(\A) \atop |\lambda|<2\sqrt{p}}U_m\left(\frac{\lambda}{2\sqrt{p}}\right)t^{m-1}
		\\
		&\quad =
		t\sum_{\lambda \in \Spec(\A) \atop |\lambda|<2\sqrt{p}}
		\left[
		\dfrac{t}{1-t^2}-\dfrac{1}{1-t^2}
		\{\log (1-\tfrac{\lambda}{\sqrt{p}}t+t^2)\}'
		\right].
	\end{align*}
	 This yields the equality
	\begin{align*}
		&(1-t^2)nt\sum_{m=1}^\infty \dfrac{a(p^m)}{2p^{m/2}}t^{m-1}
		=
		l t^2
		-t\sum_{\lambda \in \Spec(\A) \atop |\lambda|<2\sqrt{p}}
		\{\log (1-\tfrac{\lambda}{\sqrt{p}}t+t^2)\}'.
	\end{align*}
	By Ihara's formula \eqref{Ihara-formula}, the last sum 
	is described as
	\begin{align*}
		&-\sum_{\lambda \in \Spec(\A) \atop |\lambda|<2\sqrt{p}}
		\{\log (1-\tfrac{\lambda}{\sqrt{p}}t+t^2)\}'
		\\[12pt]
		&=
		\left\{
		\begin{array}{ll}
			\left(\; 
			\log \{(1-\frac{t^2}{p})^{n(p-1)/2} (1-\frac{p+1}{\sqrt{p}}t+t^2)(1+\frac{p+1}{\sqrt{p}}t+t^2)Z_{X^{p,q}}(\frac{t}{\sqrt{p}})\}\; \right)',\\
			\hspace{9cm} \textrm{if $X^{p,q}$ is bipartite},\\[12pt]
			\left(\; \log\{
			(1-\frac{t^2}{p})^{n(p-1)/2}
			(1-\frac{p+1}{\sqrt{p}}t+t^2)Z_{X^{p,q}}(\frac{t}{\sqrt{p}})\}\; \right)',
			\quad \textrm{otherwise}.
		\end{array}
		\right.
	\end{align*}
This completes the proof of the formula of $\varphi$.

From the explicit formula and the fact that $l$ is equal to $n-2$ and $n-1$ if $X^{p,q}$ is bipartite and non-bipartite, respectively, we obtain
$$\lim_{t\rightarrow 1}(t-1)\varphi(t)=-\frac{1}{2n\sqrt{p}}\frac{Z_{X^{p,q}}'}{Z_{X^{p,q}}}\left(\frac{1}{\sqrt{p}}\right)$$
and the right-hand side equals zero by a direct computation.
%
%
%
\end{proof}

We remark that Corollary \ref{cor:LPS} is also proved by
$\lim_{t\rightarrow 1}(t-1)\varphi(t)=0$ in Proposition \ref{generating}.


We have investigated the average of $\frac{N_{m}}{q^{m/2}}$ for the LPS Ramanujan graph, 
and that of Fourier coefficients of the cusp form related with the LPS Ramanujan graph.
Before ending this section, we remark
the estimate of $N_m$ itself by Huang \cite{Huang},
whose motivation is quite different from ours.

\begin{remark} 
In 2020, 
Huang established a graph-theoretic analogue of Li's criterion 
for the Ihara zeta function $Z_{G}(u)$ {\rm(}\cite[Theorem 1.3]{Huang}{\rm)}.
Namely, he studied the characterization of Ramanujan graphs
in terms of the sequence $\{ h_{m} \}_{m \geq 1}$,
which is defined as the coefficients of $\frac{d}{du}\log\Xi(q^{-1/2}u)$
with $\Xi(u)$ being a variant of the Ihara zeta function $Z_G(u)$
as in \cite[Definition 1.1]{Huang}.
More precisely, he showed that a connected $(q+1)$-regular graph $G$ 
on $n \ ( \geq 3)$ vertices is Ramanujan if and only if the inequality $h_{m} \geq 0$ holds 
for infinitely many even $m \geq 2$.
Here the sequence $\{ h_{m} \}_{m \geq 1}$ is described by \cite[Proposition 3.1]{Huang} as
\begin{align*}
h_{m} &= 
\begin{cases}
2 (n-1) + \dfrac{n e_{m} (q-1)}{q^{m/2}} + 2 T_{m}(\frac{q+1}{2 \sqrt{q}}) 
- \dfrac{N_{m}}{q^{m/2}}, & \mbox{if $G$ is non-bipartite}, \\ 
2 (n-2) + \dfrac{n e_{m} (q-1)}{q^{m/2}} + 4 e_{m} T_{m}(\frac{q+1}{2\sqrt{q}}) 
- \dfrac{N_{m}}{q^{m/2}}, & \mbox{if $G$ is bipartite.}
\end{cases} \\
\end{align*}
As a result, the estimate of $N_m$ itself for a Ramanujan graph as in \cite[Theorem 1.4]{Huang} is given by transforming the inequality $h_m\ge0$.
\end{remark}

\section*{Acknowledgements}
The authors would like to thank the anonymous referee
for careful reading and suggesting improvements of some statements. 
Takehiro Hasegawa was partially supported by JSPS KAKENHI (grant numbers 19K03400 and 22K03246).
Hayato Saigo was partially supported by JSPS KAKENHI (grant numbes 19K03608 and 22K03405) and by Research Origin for Dressed Photon.
Seiken Saito was partially supported by JSPS KAKENHI (grant numbers 19K03608 and 22K03405) and by Research Origin for Dressed Photon.
Shingo Sugiyama was partially supported by JSPS KAKENHI (grant number 20K14298).
\section*{Data Availability Statement}
No data were generated or used in the preparation
of this paper.

\section*{Declarations}
{\bf Conflict of interest}\  The authors confirm that they have no conflict of interest
in connection with this paper.


\end{document}